\documentclass[reqno,11pt]{amsart}

\usepackage{amsthm, amsmath, amssymb}
\usepackage{bm}

\usepackage[margin=1.1in]{geometry}

\usepackage{microtype}
\usepackage[T1]{fontenc}

\usepackage[hidelinks]{hyperref}

\usepackage[textsize=scriptsize,backgroundcolor=orange!5]{todonotes}

\usepackage[noabbrev,capitalize]{cleveref}
\crefname{equation}{}{}

\newtheorem{theorem}{Theorem}[section]
\newtheorem{proposition}[theorem]{Proposition}
\newtheorem{lemma}[theorem]{Lemma}

\theoremstyle{definition}
\newtheorem{definition}[theorem]{Definition}
\newtheorem{problem}[theorem]{Problem}

\newtheorem{example}[theorem]{Example}

\theoremstyle{remark}
\newtheorem*{remark}{Remark}

\newcommand{\floor}[1]{\left\lfloor #1 \right\rfloor}

\newcommand{\paren}[1]{\left( #1 \right)}

\DeclareMathOperator{\sgn}{sgn}
\newcommand{\rc}{\mathrm{rc}}

\DeclareMathOperator{\Aut}{Aut}
\DeclareMathOperator{\Sub}{Sub}

\usepackage{tikz}
\usetikzlibrary{decorations.markings}
\usetikzlibrary{arrows}

\tikzset{every picture/.append style={
scale=0.7,
baseline={([yshift=-.8ex]current bounding box.center)},
font=\small
}
}

\tikzstyle{v} = [draw,circle,inner sep=0pt, minimum width=2.5pt,fill=black]
\tikzstyle{e}=[postaction={decorate,decoration={
        markings,
        mark=at position .5*\pgfdecoratedpathlength+1.5pt 
        with {\arrow{angle 90}}
      }}]
\tikzstyle{b} = [line width=2.5pt, blue]
\tikzstyle{r} = [line width=2.5pt, red ]

\title{Impartial digraphs}

\author{Yufei Zhao}
\address{Department of Mathematics, Massachusetts Institute of Technology, Cambridge, MA, USA}
\email{yufeiz@mit.edu}

\author{Yunkun Zhou}
\address{Department of Mathematics, Stanford University, Stanford, CA, USA}
\email{yunkunzhou@stanford.edu}
\thanks{Zhao was supported by NSF Awards DMS-1362326 and DMS-1764176, the MIT Solomon Buchsbaum Fund, and a Sloan Research Fellowship. Zhou was supported by the MIT Undergraduate Research Opportunities Program}

\begin{document}

\begin{abstract}
    We prove a conjecture of Fox, Huang, and Lee that characterizes directed graphs that have constant density in all tournaments: they are disjoint unions of trees that are each constructed in a certain recursive way.
\end{abstract}

\maketitle

\section{Introduction} \label{sec:intro}

In this paper we give a complete answer to the following question (our notation: $\vec H$ with an arrow on top denotes a directed graph, or \emph{digraph}, and $H$ without the arrow denotes the underlying undirected graph of $\vec H$; we write $|H|$ for the number of vertices of $H$):
\begin{quote}
Which directed graphs $\vec H$ have the property that for some $n \ge |H|$, all $n$-vertex tournaments contain the same number of copies of $\vec H$ as subgraphs?
\end{quote}
We say that a digraph $\vec H$ is \emph{impartial} if it has the above property. Note that if the above property holds for some $n_0 \ge |H|$, then it holds for all $n \ge n_0$, due to a double-counting argument considering all $n_0$-vertex subsets of $\vec H$.

\begin{example} A single directed edge is clearly impartial. The following graph is also impartial.
\begin{equation}\label{eq:intro-ex1}
\begin{tikzpicture}
	\node[v] (1) at (0,0) {};
	\node[v] (2) at (1,0) {};
	\node[v] (3) at (2,0) {};
	\node[v] (4) at (3,0) {};
	\draw[e] (1) to (2);
	\draw[e] (3) to (2);
	\draw[e] (4) to (3);
\end{tikzpicture}
\end{equation}
Indeed, all tournaments on $n$ vertices have the same number of copies of
\[
\begin{tikzpicture}
	\node[v] (1) at (0,0) {};
	\node[v] (2) at (1,0) {};
	\node[v] (3) at (2,0) {};
	\node[v] (4) at (3,0) {};
	\draw[e] (1) to (2);
	\draw[e] (4) to (3);
\end{tikzpicture}
\]
and each such subgraph uniquely extends to a copy of \cref{eq:intro-ex1}.

We can continue this recursive argument and deduce that
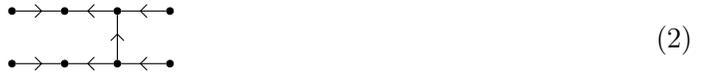
\begin{equation}\label{eq:intro-ex2}
\begin{tikzpicture}
	\node[v] (1) at (0,0) {};
	\node[v] (2) at (1,0) {};
	\node[v] (3) at (2,0) {};
	\node[v] (4) at (3,0) {};
	\node[v] (1a) at (0,1) {};
	\node[v] (2a) at (1,1) {};
	\node[v] (3a) at (2,1) {};
	\node[v] (4a) at (3,1) {};
	\draw[e] (1) to (2);
	\draw[e] (3) to (2);
	\draw[e] (4) to (3);
	\draw[e] (1a) to (2a);
	\draw[e] (3a) to (2a);
	\draw[e] (4a) to (3a);
	\draw[e] (3)--(3a);
\end{tikzpicture}
\end{equation}
is impartial. Indeed, the digraph obtained by removing the middle edge from \cref{eq:intro-ex2} is impartial by earlier observations, and any copy \cref{eq:intro-ex2} without the middle edge in a tournament can be uniquely extended to a copy of \cref{eq:intro-ex2}.
\end{example}

The above argument generalizes to all directed graphs that can be constructed in the following recursive manner.

\begin{definition} \label{def:rbm}
We say that a digraph $\vec T$ is \emph{recursively bridge-mirrored} if it can be constructed recursively in the following manner:
\begin{itemize}
    \item The tree with one vertex is recursively bridge-mirrored;
    \item Suppose $\vec T$ is recursively bridge-mirrored. Mark an arbitrary vertex of $\vec T$ as its root, and create a new graph by taking two identical copies of this rooted $\vec T$ and adding a new directed edge from one root to the other. The resulting digraph is also said to be recursively bridge-mirrored.
\end{itemize}
We say that an undirected graph $T$ is recursively bridge-mirrored if it can be constructed as above but forgetting edge orientations.
\end{definition}

\begin{example}
	Here is a sequence of recursively bridge-mirrored digraphs constructed recursively as in the definition.
	
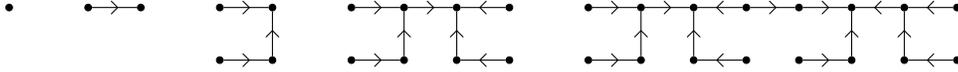
\[
	\begin{tikzpicture}
		\begin{scope}
			\node[v] (0) at (0,1) {};
		\end{scope}
		
		\begin{scope}[xshift=1.5cm]
			\node[v] (0) at (0,1) {};
			\node[v] (1) at (1,1) {};
			\draw[e] (0)--(1);				
		\end{scope}
		
		\begin{scope}[xshift=4cm]
			\node[v] (0) at (0,1) {};
			\node[v] (1) at (1,1) {};
			\node[v] (2) at (1,0) {};
			\node[v] (3) at (0,0) {};
			\draw[e] (0)--(1);				
			\draw[e] (2)--(1);
			\draw[e] (3)--(2);
		\end{scope}
		
		\begin{scope}[xshift=6.5cm]
			\node[v] (0) at (0,1) {};
			\node[v] (1) at (1,1) {};
			\node[v] (2) at (1,0) {};
			\node[v] (3) at (0,0) {};
			\node[v] (4) at (3,1) {};
			\node[v] (5) at (2,1) {};
			\node[v] (6) at (2,0) {};
			\node[v] (7) at (3,0) {};

			\draw[e] (0)--(1);				
			\draw[e] (2)--(1);
			\draw[e] (3)--(2);
			\draw[e] (4)--(5);				
			\draw[e] (6)--(5);
			\draw[e] (7)--(6);
			\draw[e] (1)--(5);
		\end{scope}
		
		\begin{scope}[xshift=11cm]
		
			\node[v] (0) at (0,1) {};
			\node[v] (1) at (1,1) {};
			\node[v] (2) at (1,0) {};
			\node[v] (3) at (0,0) {};
			\node[v] (4) at (3,1) {};
			\node[v] (5) at (2,1) {};
			\node[v] (6) at (2,0) {};
			\node[v] (7) at (3,0) {};
			\node[v] (0a) at (7,1) {};
			\node[v] (1a) at (6,1) {};
			\node[v] (2a) at (6,0) {};
			\node[v] (3a) at (7,0) {};
			\node[v] (4a) at (4,1) {};
			\node[v] (5a) at (5,1) {};
			\node[v] (6a) at (5,0) {};
			\node[v] (7a) at (4,0) {};

			\draw[e] (0)--(1);				
			\draw[e] (2)--(1);
			\draw[e] (3)--(2);
			\draw[e] (4)--(5);				
			\draw[e] (6)--(5);
			\draw[e] (7)--(6);
			\draw[e] (1)--(5);
			\draw[e] (0a)--(1a);				
			\draw[e] (2a)--(1a);
			\draw[e] (3a)--(2a);
			\draw[e] (4a)--(5a);				
			\draw[e] (6a)--(5a);
			\draw[e] (7a)--(6a);
			\draw[e] (1a)--(5a);
			\draw[e] (4)--(4a);

		\end{scope}

	\end{tikzpicture}
	\]
\end{example}

All recursively bridge-mirrored graphs are trees. The earlier argument easily generalizes to show that all recursively bridge-mirrored digraphs are impartial. Our main result proves the converse.

\begin{theorem} \label{thm:main}
A directed graph is impartial if and only if it is a disjoint union of recursively bridge-mirrored digraphs.
\end{theorem}

\begin{remark}
It then follows that if $\vec H$ is impartial, then for \emph{every} $n$, all $n$-vertex tournaments $\vec K$ have the same number of copies of $\vec H$.
\end{remark}

\subsection*{Background and motivation.}
The first author learned of this problem from Jacob Fox, who had proposed it together with Hao Huang and Choongbum Lee around 2013 when they considered density problems for digraphs and tournaments, including variants of Sidorenko's conjecture and the inducibility problem, where one wishes to maximize or minimize the number of copies of a fixed digraph in a tournament. They formulated \cref{thm:main} as a conjecture after computer experiments, and proved that every impartial digraph must be a disjoint union of trees whose number of vertices is a power of 2; see \cref{prop:tree}. (The names ``impartial'' and ``recursively bridge-mirrored'' are ours.)

For undirected graphs, conjectures of Sidorenko~\cite{Sid93} and Erd\H{o}s--Simonovits~\cite{ES83} (commonly referred to as \emph{Sidorenko's conjecture}) say that for every bipartite graph $H$, the $H$-density in a graph of fixed density is minimized asymptotically by a random graph. Lately the conjecture has been proved for many families of $H$ \cite{CFS10,CKLL18,CL-blowup,CL17,Hat10,KLL16,LS,SzeSidorenko}, though the conjecture remains open in general. In particular, the case $H = K_{5,5}\setminus C_{10}$ is open.

Our main theorem can be viewed as an ``equality case'' of the directed analog of Sidorenko's conjecture, which turns out to be already quite intricate. In \cref{sec:further}, we propose several directions that are worth further investigation, including ``positive'' and ``negative'' digraphs, as well as generalizations to hypergraphs.

\subsection*{Outline}
Our proof of \cref{thm:main} proceeds in several steps:
\begin{itemize}
	\item By comparing the density of $\vec H$ in a transitive tournament and a random tournament, along with some integer divisibility considerations, $H$ is shown to be a forest (\cref{sec:forest}).
	\item Using a graph limit argument, we convert impartiality to a polynomial identity (\cref{sec:limit}), and equating the coefficients leads us to subgraph statistics of impartial graphs (\cref{sec:odd}).
	\item Then, we establish that the undirected structure of $\vec H$ is recursively bridge-mirrored (\cref{sec:undirected}). 
	\item Once we have the undirected structure, we establish the directed structure of $\vec H$ (\cref{sec:directed}), and this step requires some additional analysis of undirected recursively bridge-mirrored trees (\cref{sec:property}).
\end{itemize}

\section{Reduction to forests} \label{sec:forest}

We begin with a short argument due to Fox, Huang, and Lee (unpublished) that proves that the undirected structure of an impartial digraph is a forest.

By a \emph{component} of a digraph we mean a weakly connected component (so its meaning is compatible with the undirected structure).

\begin{proposition} \label{prop:tree}
Every component of an impartial digraph is a tree whose number of vertices is a power of $2$.
\end{proposition}

\begin{proof}
Let $\vec H$ be an impartial digraph, and $n\ge |H|$ such that every tournament on $n$ vertices has exactly the same number of copies of $\vec H$.

Suppose that $\vec H$ has $k$ components, and the $i$-th component has $n_i$ vertices, $m_i$ edges, and exactly $\ell_i$ linear extensions. The number of labeled copies of $\vec H$ in a transitive tournament on $n$ vertices is exactly
\[
\binom{n}{n_1, n_2, \dots, n_k, n-n_1-\cdots - n_k} \ell_1\cdots \ell_k.
\]
On the other hand, the expected number of labeled copies of $\vec H$ in a uniform random tournament on $n$ vertices is, by linearity of expectations,
\[
\binom{n}{n_1, n_2, \dots, n_k, n-n_1-\cdots - n_k} 2^{-m_1-\cdots -m_k} n_1! \cdots n_k!.
\]
Since $\vec H$ is impartial, the above two quantities are equal, and hence
\[
2^{m_1 + \cdots + m_k} \ell_1 \cdots \ell_k = n_1! \cdots n_k!.
\]
Since each component $H_i$ is connected, we have $m_i \ge n_i - 1$. On the other hand, the highest power of $2$ that divides $n_i!$ has exponent $\floor{n_i/2} + \floor{n_i/4} + \cdots \le n_i - 1$, with equality if and only if $n_i$ is a power of $2$. Since $2^{-m_1-\cdots -m_k} n_1! \cdots n_k! =  \ell_1 \cdots \ell_k $ is an integer, for every $i$, $n_i$ is a power of 2 and $m_i = n_i - 1$, so every component is a tree.
\end{proof}

\section{A graph limit argument} \label{sec:limit}

In this section, using tools from graph limit theory, we deduce some internal properties of an impartial digraph.

Adapting the notion of a \emph{graphon} from graph limit theory~\cite{Lov}, we define a \emph{tourneyon} to be a measurable function $W \colon [0,1]^2 \to [0,1]$ satisfying $W(x,y) + W(y,x) = 1$ for all $x,y \in [0,1]$. Just as graphons are the analytic limit objects for sequences of graphs, tourneyons are the limit objects for sequences of tournaments. Limits of tournaments have also appeared in \cite{Tho18}; also see \cite{DJ08,LS10} for related concepts.

We write $V(H)$ for the vertex set and $E(\vec H)$ the edge set of $\vec H$. Directed edges are given by ordered pairs $(u,v)$. 
 
We define the \emph{density} of a digraph $\vec H$ in $W$ to be
\[
t(\vec H, W) := \int_{[0,1]^{V(H)}} \prod_{(u,v) \in E(\vec H)} W(x_u, x_v) \, \prod_{v \in V(H)} dx_v.
\]

\begin{example}
$
t(
\begin{tikzpicture}[scale=.6]
	\node[v] (0) at (0,0) {};
	\node[v] (1) at (1,0) {};
	\node[v] (2) at (.5,.87) {};
	\draw[e] (0)--(2);
	\draw[e] (0)--(1);
	\draw[e] (1)--(2);
\end{tikzpicture},
W) = \int_{[0,1]^3} W(x,y)W(x,z)W(y,z) \, dxdydz.
$	
\end{example}

\begin{proposition} \label{prop:tourneyon}
A digraph $\vec H$ is impartial if and only if for all tourneyons $W$,
\[
	t(\vec H, W) = 2^{-|E(H)|}.
\]
\end{proposition}

Let us first prove the ``only if'' direction (the relevant direction for our main theorem) and defer the proof of the ``if'' direction until a bit later.

\begin{proof} (``Only if'' direction)
By standard arguments in graph limit theory \cite[Chapter 10]{Lov}, for each tourneyon $W$ there exists a sequence of $n$-vertex tournaments that approaches $W$ in the cut metric as $n$ grows to infinity, and thus by the counting lemma, they have $\vec H$-densities approaching $t(\vec H, W)$. Hence by impartiality, $t(\vec H, W)$ does not depend on the tourneyon $W$. 
The conclusion follows from noting that the constant tourneyon $W \equiv 1/2$ has $H$-density $2^{-|E(H)|}$.
\end{proof}

For every digraph $\vec H$ and positive integer $n$, define the following polynomial with variables $\bm a = (a_1, \dots, a_n)$ and $\bm b = (b_{ij})_{1 \le i < j \le n}$ associated to $\vec H$:
\begin{equation}\label{eq:poly-expand}
P_{\vec H,n} (\bm a; \bm b)
:=\sum_{\pi \colon V(H) \to [n]}
\prod_{v \in V(H)} a_{\pi(v)}
\prod_{(u,v) \in E(\vec H)} (1 + b_{\pi(u)\pi(v)}).
\end{equation}
where we set $b_{ij} = -b_{ji}$ if $i > j$ and $b_{ii} = 0$ for all $i$.

For $a_i \in [0,1]$ and $b_{ij} \in [-1,1]$ for all $i,j$ and satisfying $a_1+\cdots + a_n = 1$, the tourneyon $W$ obtained by partitioning $[0,1]$ into intervals $I_1, \dots, I_n$ of lengths $a_1, \dots, a_n$, and setting $W(x,y) = (1 + b_{ij})/2$ whenever $(x,y) \in I_i \times I_j$, satisfies
\begin{equation}\label{eq:associated-tourneyon}
t(\vec H, W) = 2^{-e(H)} P_{\vec H,n} (\bm a; \bm b).
\end{equation}

\begin{lemma} \label{lem:poly-tourney}
Let $\vec H$ be a digraph and $\vec K$ a tournament on $n$ vertices. Then the number of labeled copies of $\vec H$ in $\vec K$ equals the sum of the coefficients of all $a_{i_1}\cdots a_{i_{|H|}}$ with distinct $i_1, \dots, i_{|H|} \in [n]$ in the polynomial $2^{-e(H)} P_{\vec H,n}(\bm a; \bm b)$ evaluated at $b_{ij} = 1$ if $(i,j) \in E(\vec K)$ and $b_{ij} = -1$ if $(j,i) \in E(\vec K)$ for all $i<j$.
\end{lemma}

\begin{proof}
With the above $\bm b$-values, every edge-orientation respecting homomorphism $\pi \colon \vec H \to \vec K$ corresponds to a nonzero term on the right-hand side of \cref{eq:poly-expand}, and the map is injective if and only if the $\bm a$ factor is square-free.
\end{proof}

\begin{proposition}\label{prop:poly}
Let $\vec H$ be a digraph and $n \ge |H|$. Then $\vec H$ is impartial if and only if
\begin{equation}\label{eq:poly-id}
P_{\vec H,n}(\bm a; \bm b) = (a_1 + \cdots + a_n)^{|H|}.
\end{equation}
as a polynomial identity in the variables $\bm a$ and $\bm b$.
\end{proposition}
\begin{proof}
If $\vec H$ is impartial, then by \cref{prop:tourneyon} and \cref{eq:associated-tourneyon}, one has \cref{eq:poly-id} for all $a_i \in [0,1]$ and $b_{ij} \in [-1,1]$ for all $i,j$ and satisfying $a_1+\cdots + a_n = 1$. Since \cref{eq:poly-id} is homomgenous in $\bm a$, it must be a polynomial identity.

Conversely, suppose \cref{eq:poly-id} holds. So $P_{\vec H,n}(\bm a; \bm b)$ does not depend on $\bm b$, and by \cref{lem:poly-tourney}, the number of copies of $\vec H$ is constant in all $n$-vertex tournaments. Hence $\vec H$ is impartial.
\end{proof}

\begin{proof}[Proof of the ``if'' direction of \cref{prop:tourneyon}]
	If $t(\vec H, W) = 2^{-|E(H)|}$ for all tourneyons, then setting $W$ as in \cref{eq:associated-tourneyon}, we see that $\vec H$ satisfies the polynomial identity \cref{eq:poly-id} for all $n$, and hence $\vec H$ is impartial by \cref{prop:poly}.
\end{proof}

Here is an immediate consequence of the algebraic characterization of impartiality.

\begin{proposition}
\label{prop:components}
A digraph is impartial if and only if all its weakly connected components are impartial.
\end{proposition} 

\begin{proof}
The ``if'' direction is clear from definition (by embedding one component at a time).
For the ``only if'' direction, let $\vec H$ be an impartial digraph with components $\vec H_1, \dots, \vec H_\ell$.
Since $t(\vec H, W) = t(\vec H_1, W) \cdots t(\vec H_\ell, W)$ for every tourneyon $W$, one has $P_{\vec H,n} = P_{\vec H_1,n} \cdots P_{\vec H_\ell,n}$. In particular, since $\vec H$ is impartial, by \cref{prop:poly} one has the polynomial identity with $n = |H|$
\[
\prod_{j=1}^\ell P_{\vec H_j,n}(\bm a; \bm b) 
= (a_1 + \cdots + a_n)^{|H_1| + \cdots + |H_\ell|}.
\]
By unique factorization of polynomials, along with noting the degree of $\bm a$ and checking constant factors, we have
\[
P_{\vec H_j,n}(\bm a; \bm b) 
= (a_1 + \cdots + a_n)^{|H_j|}
\]
for each $j=1,\dots, \ell$, and thus by \cref{prop:poly} again, every $\vec H_j$ is impartial.
\end{proof}

\section{Odd automorphisms and odd graphs} \label{sec:odd}

In this section we derive some subgraph statistics of impartial digraphs from the polynomial characterization in the previous section.

Given digraphs $\vec F$ and $\vec H$ along with an undirected graph homomorphism $\sigma \colon F \to H$ (i.e., a map $V(F) \to V(H)$ that carries every edge of $F$ to an edge of $H$), we define $\sgn(\sigma; \vec F, \vec H) = (-1)^r$ where $r = |\{(u,v) \in E(\vec F) : (\sigma(v), \sigma(u)) \in E(\vec H)\}|$ is the number of edges of $\vec F$ whose orientation is reversed under the map $\sigma$.

Let $\Aut(H)$ be the group of automorphisms of an undirected graph $H$. For $\sigma \in \Aut(H)$, write $\sgn(\sigma) = \sgn(\sigma; \vec H, \vec H)$ where $\vec H$ is an arbitrary orientation of the edges of $H$. Note that $\sgn(\sigma)$ does not depend on this choice of edge orientations, since reversing the orientation of a single edge of $\vec H$ does not change the sign. We say that the automorphism $\sigma$ of $H$ is \emph{odd} if $\sgn(\sigma) = -1$ and \emph{even} if $\sgn(\sigma) = 1$. Note that $\sgn \colon \Aut(H) \to \{1, -1\}$ is a group homomorphism.

We call a graph or a digraph \emph{odd} if its underlying undirected graph has an odd automorphism, and \emph{even} if it has no odd automorphism. The classification of graphs as even or odd plays an important role in our arguments. Note a graph is odd if and only if it has an odd component.

\begin{example}
The horizontal reflection (i.e., across a vertical axis) of the following odd graph gives an odd automorphism, while the vertical reflection gives an even automorphism.
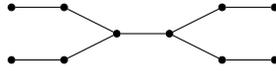
\[
\begin{tikzpicture}
\node[v] (0) at (2,0.5) {};	
\node[v] (1) at (1,1) {};	
\node[v] (2) at (0,1) {};	
\node[v] (3) at (1,0) {};	
\node[v] (4) at (0,0) {};	
\node[v] (5) at (3,0.5) {};	
\node[v] (6) at (4,1) {};	
\node[v] (7) at (5,1) {};	
\node[v] (8) at (4,0) {};	
\node[v] (9) at (5,0) {};	
\draw (0)--(1)--(2) (0)--(3)--(4) (0)--(5)--(6)--(7) (5)--(8)--(9);
\end{tikzpicture}
\]
\end{example}

Given a pair of even digraphs $\vec H_1$ and $\vec H_2$ with $H_1 \cong H_2$ (meaning that they have isomorphic undirected structures), define $\sgn(\vec H_1, \vec H_2) := \sgn(\sigma; \vec H_1, \vec H_2)$ where $\sigma \colon  H_1 \to H_2$ is some isomorphism of the undirected structures. Since $H_1 \cong H_2$ is even,  $\sgn(\vec H_1, \vec H_2)$ does not depend on the choice of the isomorphism $\sigma$. We do not define $\sgn(\vec H_1, \vec H_2)$ unless both digraphs are even and have isomorphic undirected structures.

\begin{example}
	$\sgn(
	\begin{tikzpicture}[scale=.7]
		\node[v] (0) at (0,0) {};
		\node[v] (1) at (1,0) {};
		\node[v] (2) at (2,0) {};
		\draw[e] (0)--(1);
		\draw[e] (1)--(2);
	\end{tikzpicture}
	,
	\begin{tikzpicture}[scale=.7]
		\node[v] (0) at (0,0) {};
		\node[v] (1) at (1,0) {};
		\node[v] (2) at (2,0) {};
		\draw[e] (0)--(1);
		\draw[e] (2)--(1);
	\end{tikzpicture}
	) = -1
	$ and 
	$\sgn(
	\begin{tikzpicture}[scale=.7]
		\node[v] (0) at (0,0) {};
		\node[v] (1) at (1,0) {};
		\node[v] (2) at (2,0) {};
		\draw[e] (0)--(1);
		\draw[e] (1)--(2);
	\end{tikzpicture}
	,
	\begin{tikzpicture}[scale=.7]
		\node[v] (0) at (0,0) {};
		\node[v] (1) at (1,0) {};
		\node[v] (2) at (2,0) {};
		\draw[e] (1)--(0);
		\draw[e] (2)--(1);
	\end{tikzpicture}
	) = 1
	$
\end{example}

Given undirected graphs $F$ and $H$, we denote the set of subgraphs of $H$ isomorphic to $F$ by
\[
\Sub_F(H) := \{ F' \subseteq H : F' \cong F\}.
\]
These are unlabeled copies of $F$ in $H$. To be precise, $F' \subseteq H$ means that $V(F') \subseteq V(H)$ and $E(F') \subseteq E(H)$. In our applications, $|F| = |H|$, so that we are considering the collection of edge-subsets of $H$ that are isomorphic to $F$.

Given a digraph $\vec H$, we write 
\[
\Sub_F(\vec H) := \{ \vec F' \subseteq \vec H: F' \cong F\}.
\]
to denote the set of subgraphs of $\vec H$ (with the inherited edge-orientations) whose undirected structure is isomorphic to $F$.

\begin{example}\label{ex:sub}
Consider
\[
\vec H = 
\begin{tikzpicture}
\node[v] (0) at (0,0) {};
\node[v] (1) at (1,0) {};
\node[v] (2) at (2,0) {};
\node[v] (3) at (3,0) {};
\node[v] (4) at (0,1) {};
\node[v] (5) at (1,1) {};
\node[v] (6) at (2,1) {};
\node[v] (7) at (3,1) {};
\draw[e] (0)--(4);
\draw[e] (1)--(5);
\draw[e] (2)--(6);
\draw[e] (7)--(3);
\draw[e] (4)--(5);
\draw[e] (5)--(6);
\draw[e] (6)--(7);
\end{tikzpicture}
\qquad \text{and} \qquad
F = 
\begin{tikzpicture}
\node[v] (0) at (0,0) {};
\node[v] (1) at (1,0) {};
\node[v] (2) at (2,0) {};
\node[v] (3) at (3,0) {};
\node[v] (4) at (0,1) {};
\node[v] (5) at (1,1) {};
\node[v] (6) at (2,1) {};
\node[v] (7) at (3,1) {};
\draw (0)--(4);
\draw (1)--(5);
\draw (4)--(5);
\draw (5)--(6);
\end{tikzpicture}
\]
Then $\Sub_{F}(\vec H)$ has cardinality 4 and consists of (we draw the vertex set of the subgraph in the same way as $\vec H$ above)
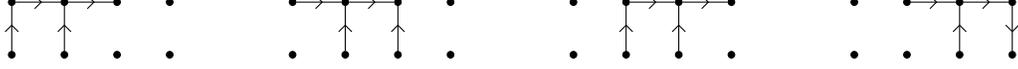
\[
\begin{tikzpicture}
\node[v] (0) at (0,0) {};
\node[v] (1) at (1,0) {};
\node[v] (2) at (2,0) {};
\node[v] (3) at (3,0) {};
\node[v] (4) at (0,1) {};
\node[v] (5) at (1,1) {};
\node[v] (6) at (2,1) {};
\node[v] (7) at (3,1) {};
\draw[e] (0)--(4);
\draw[e] (1)--(5);
\draw[e] (4)--(5);
\draw[e] (5)--(6);
\end{tikzpicture}
\hspace{4em}
\begin{tikzpicture}
\node[v] (0) at (0,0) {};
\node[v] (1) at (1,0) {};
\node[v] (2) at (2,0) {};
\node[v] (3) at (3,0) {};
\node[v] (4) at (0,1) {};
\node[v] (5) at (1,1) {};
\node[v] (6) at (2,1) {};
\node[v] (7) at (3,1) {};
\draw[e] (1)--(5);
\draw[e] (2)--(6);
\draw[e] (4)--(5);
\draw[e] (5)--(6);
\end{tikzpicture}
\hspace{4em}
\begin{tikzpicture}
\node[v] (0) at (0,0) {};
\node[v] (1) at (1,0) {};
\node[v] (2) at (2,0) {};
\node[v] (3) at (3,0) {};
\node[v] (4) at (0,1) {};
\node[v] (5) at (1,1) {};
\node[v] (6) at (2,1) {};
\node[v] (7) at (3,1) {};
\draw[e] (1)--(5);
\draw[e] (2)--(6);
\draw[e] (5)--(6);
\draw[e] (6)--(7);
\end{tikzpicture}
\hspace{4em}
\begin{tikzpicture}
\node[v] (0) at (0,0) {};
\node[v] (1) at (1,0) {};
\node[v] (2) at (2,0) {};
\node[v] (3) at (3,0) {};
\node[v] (4) at (0,1) {};
\node[v] (5) at (1,1) {};
\node[v] (6) at (2,1) {};
\node[v] (7) at (3,1) {};
\draw[e] (2)--(6);
\draw[e] (7)--(3);
\draw[e] (5)--(6);
\draw[e] (6)--(7);
\end{tikzpicture}
\]
\end{example}

For vectors $\bm a = (a_1, \dots, a_n)$ and $\bm b = (b_{ij})_{i<j}$ (setting $b_{ii} = 0$ and $b_{ji} = -b_{ij}$), we write
\begin{equation}\label{eq:ab}
a_{F} := \prod_{i \in V(F)} a_i \quad \text{and} \quad b_{\vec F} :=\prod_{(i,j) \in E(\vec F)} b_{ij}.
\end{equation}
Recall the polynomial $P_{\vec H ,n}(\bm a; \bm b)$ from \cref{eq:poly-expand} and the relation $b_{ij} = - b_{ji}$. Expanding, we obtain
\[
P_{\vec H,n} (\bm a; \bm b)
= \sum_{\pi \colon V(H) \to [n]} \sum_{\substack{\vec F \subseteq \vec H: \\ V(F) = V(H)}} 
\prod_{v \in V(H)} a_{\pi(v)}
\prod_{(u,v) \in E(\vec F)} b_{\pi(u)\pi(v)}.
\]

\begin{lemma} \label{lem:poly-coeff-sum}
Let $\vec H$ and $\vec F$ be digraphs both with vertex set $[k]$. The coefficient of $a_F b_{\vec F}$ in $P_{\vec H ,n}(\bm a; \bm b)$ with $n \ge k$ is zero if $F$ is odd, and
\[
|\Aut F| \sum_{\vec F' \in \Sub_F(\vec H)} \sgn(\vec F, \vec F')
\]
if $F$ is even.
\end{lemma}

\begin{proof}
The coefficient of $a_F b_{\vec F}$ in  $P_{\vec H,n}(\bm a; \bm b)$ equals to the sum of $\sgn(\pi; \vec F, \vec H)$ over all bijections $\pi \colon V(F) \to V(H)$ that induce homomorphisms $F \to H$. We can then partition this sum by orbits of the automorphisms group of $F$, and the conclusion then follows.
\end{proof}

\begin{example}
	Continuing \cref{ex:sub}, letting $\vec F$ denote the orientation of $F$ inherited from $\vec H$, we find that $\sgn(\vec F, \vec F')$ for the four listed elements $\vec F' \in \Sub_F(\vec H)$ are $1, 1, 1, -1$.
\end{example}

The following subgraph statistics characterization of impartiality is the main property that we need in the rest of the proof of the main theorem.

\begin{proposition} \label{prop:sgn-sum}
A digraph $\vec H$ is impartial if and only if for every even digraph $\vec F$ on $|H|$ vertices with at least one edge,
\begin{equation}\label{eq:sgn-sum}
\sum_{\vec F' \in \Sub_F(\vec H)} \sgn(\vec F, \vec F') = 0.
\end{equation}
\end{proposition}

\begin{proof}
Suppose $\vec H$ is impartial.
By \cref{eq:poly-id}, the coefficient of $a_F b_{\vec F}$ in $P_{\vec H}$ is zero provided that $F$ has at least one edge, and by \cref{lem:poly-coeff-sum} we obtain \cref{eq:sgn-sum}.

Conversely, suppose \cref{eq:sgn-sum} holds for all even $\vec F$ with at least one edge, so that the coefficient of $a_F b_{\vec F}$ in $P_{\vec H}$ is zero by \cref{lem:poly-coeff-sum}.  It follows that the polynomial identity \cref{eq:poly-id} holds term-by-term for all terms that are square-free in $\bm a$, and thus by \cref{lem:poly-tourney}, $\vec H$ has the same number of copies in all $n$-vertex tournaments, and hence $\vec H$ is impartial.
\end{proof}

\section{Undirected structure} \label{sec:undirected}

The main goal of this section is to establish the undirected structure of an impartial digraph.

\begin{proposition} \label{prop:undirected}
The undirected structure of every weakly connected impartial digraph is a recursively bridge-mirrored tree.
\end{proposition}

We saw from \cref{prop:tree} that it suffices to consider trees. In this section we focus on undirected structure, and by default ``tree'' means undirected tree.

\begin{definition}[Mirror-bridge]
An edge $f$ of a tree $T$ is called a \emph{mirror-bridge} if $T$ has an automorphism that swaps the two endpoints of $f$. In this case, we call the two isomorphic components of $T\setminus f$ ($T$ with edge $f$ removed but all vertices kept) \emph{half-branches} of $T$, 
\end{definition}

\begin{example}
The mirror-bridge of the first tree below is highlighted, while the second tree does not have a mirror-bridge.
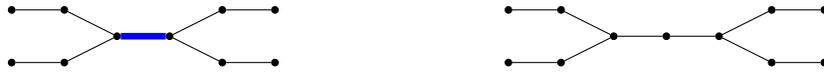
\[
\begin{tikzpicture}
\node[v] (0) at (2,0.5) {};	
\node[v] (1) at (1,1) {};	
\node[v] (2) at (0,1) {};	
\node[v] (3) at (1,0) {};	
\node[v] (4) at (0,0) {};	
\node[v] (5) at (3,0.5) {};	
\node[v] (6) at (4,1) {};	
\node[v] (7) at (5,1) {};	
\node[v] (8) at (4,0) {};	
\node[v] (9) at (5,0) {};	
\draw (0)--(1)--(2) (0)--(3)--(4) (5)--(6)--(7) (5)--(8)--(9);
\draw[b] (0)--(5);
\end{tikzpicture}
\hspace{3cm}
\begin{tikzpicture}
\node[v] (0) at (2,0.5) {};	
\node[v] (1) at (1,1) {};	
\node[v] (2) at (0,1) {};	
\node[v] (3) at (1,0) {};	
\node[v] (4) at (0,0) {};	
\node[v] (5) at (3,0.5) {};	
\node[v] (5a) at (4,0.5) {};	
\node[v] (6) at (5,1) {};	
\node[v] (7) at (6,1) {};	
\node[v] (8) at (5,0) {};	
\node[v] (9) at (6,0) {};	
\draw (0)--(1)--(2) (0)--(3)--(4) (0)--(5)--(5a)--(6)--(7) (5a)--(8)--(9);
\end{tikzpicture}
\]
This is a half-branch of the first tree:
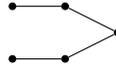
\[
\begin{tikzpicture}
\node[v] (0) at (2,0.5) {};	
\node[v] (1) at (1,1) {};	
\node[v] (2) at (0,1) {};	
\node[v] (3) at (1,0) {};	
\node[v] (4) at (0,0) {};	
\draw (0)--(1)--(2) (0)--(3)--(4);
\end{tikzpicture}
\]
\end{example}

\begin{lemma}\label{lem:mirror-bridge}
A tree is odd if and only if it has a mirror-bridge.
\end{lemma}

\begin{proof}
An automorphism of a tree that reverses a mirror-bridge must be odd.

Conversely, suppose a tree $T$ has an odd automorphism $\pi$. We will construct subtrees $T_0 \supsetneq T_1 \supsetneq \cdots \supsetneq T_k$ iteratively such that $\pi$ induces an automorphism on each $T_i$. 

Let $T_0 = T$. For each $i \ge 0$, if $T_i$ is not a single vertex or an edge, then let  $L_i$ denote the set of leaves of $T_i$, and set $T_{i+1}$ to be $T_i$ with all its leaves and pendent edges removed. When the process terminates at $i = k$, $T_k$ is either a single vertex or an edge. Since an automorphism preserves leaves, an easy induction argument shows that $\pi$ induces an automorphism for every $T_i$.

Make $T$ into a directed graph by orienting $T_k$ arbitrarily and orienting all edges from $L_i$ to $L_{i+1}$ for every $i$. Note that $\pi$ preserves the orientations of all edges except possibly $T_k$. Since $\pi$ is odd, $T_k$ must be a single edge and reversed by $\pi$, so that it is a mirror-bridge of $T$.
\end{proof}

\begin{definition}[Branch]
Let $T$ be a tree and $uv$ an edge of $T$. Let $B$ be the connected component of $T \setminus uv$ containing $v$. Then the rooted tree $(B, v)$ (or just $B$ itself) is called a \emph{branch of $T$ cut from $uv$}.
\end{definition}

\begin{example}
The branch $(B, v)$ of the tree $T$ below cut from $uv$ is highlighted.
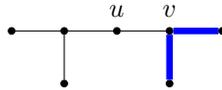
\[
\begin{tikzpicture}
\node[v] (0) at (1,0) {};	
\node[v] (1) at (0,1) {};	
\node[v] (2) at (1,1) {};	
\node[v,label=above:$u$] (3) at (2,1) {};	
\node[v,label=above:$v$] (4) at (3,1) {};	
\node[v] (5) at (4,1) {};	
\node[v] (6) at (3,0) {};	
\draw (1)--(2)--(3)--(4) (0)--(2);
\draw[b] (6)--(4)--(5);
\end{tikzpicture}
\]
\end{example}

The following procedure produces a canonical even subgraph of a given forest.

\begin{definition}[Recursive cutting] \label{def:rc}
	Let $F$ be a forest. The \emph{recursive cutting} procedure applied to $F$ produces a sequence $(F_0, F_1, \dots)$ where $F_0 = F$, and each $F_{i+1}$ is obtained from $F_i$ from removing the bridge from each odd component of $F_i$. Let $F_\rc$ denote the even forest that the sequence eventually stabilizes to.
	
	For a digraph $\vec F$, we write $\vec F_\rc$ for $F_\rc$ with edge-orientations inherited from $\vec F$.
\end{definition}

\begin{example}
	Here is an example of recursive cutting:
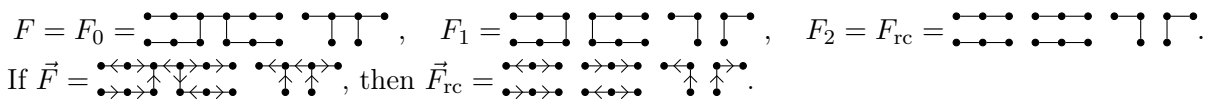
\[
F = F_0 =
\begin{tikzpicture}[scale=.5]
\begin{scope}
\foreach \x in {0,...,5}
	\foreach \y in {0,1}
		\node[v] (\x\y) at (\x,\y) {};
\draw (00)--(10)--(20)--(21)--(11)--(01);
\draw (50)--(40)--(30)--(31)--(41)--(51);
\draw (31)--(21);
\end{scope}
\begin{scope}[xshift=5cm]
\node[v] (20) at (2,0) {};
\node[v] (30) at (3,0) {};
\node[v] (11) at (1,1) {};
\node[v] (21) at (2,1) {};
\node[v] (31) at (3,1) {};
\node[v] (41) at (4,1) {};
\draw (20)--(21)--(11);
\draw (30)--(31)--(41);
\draw (31)--(21);
\end{scope}
\end{tikzpicture}\ , \quad
F_1 =
\begin{tikzpicture}[scale=.5]
\foreach \x in {0,...,5}
	\foreach \y in {0,1}
		\node[v] (\x\y) at (\x,\y) {};
\draw (00)--(10)--(20)--(21)--(11)--(01);
\draw (50)--(40)--(30)--(31)--(41)--(51);
\begin{scope}[xshift=5cm]
\node[v] (20) at (2,0) {};
\node[v] (30) at (3,0) {};
\node[v] (11) at (1,1) {};
\node[v] (21) at (2,1) {};
\node[v] (31) at (3,1) {};
\node[v] (41) at (4,1) {};
\draw (20)--(21)--(11);
\draw (30)--(31)--(41);
\end{scope}
\end{tikzpicture}\ , \quad
F_2 = F_\rc =
\begin{tikzpicture}[scale=.5]
\foreach \x in {0,...,5}
	\foreach \y in {0,1}
		\node[v] (\x\y) at (\x,\y) {};
\draw (00)--(10)--(20)(21)--(11)--(01);
\draw (50)--(40)--(30)(31)--(41)--(51);
\begin{scope}[xshift=5cm]
\node[v] (20) at (2,0) {};
\node[v] (30) at (3,0) {};
\node[v] (11) at (1,1) {};
\node[v] (21) at (2,1) {};
\node[v] (31) at (3,1) {};
\node[v] (41) at (4,1) {};
\draw (20)--(21)--(11);
\draw (30)--(31)--(41);
\end{scope}
\end{tikzpicture}.
\]
If 
$
\vec F = 
\begin{tikzpicture}[scale=.5]
\foreach \x in {0,...,5}
	\foreach \y in {0,1}
		\node[v] (\x\y) at (\x,\y) {};
\draw[e] (00)--(10);
\draw[e] (10)--(20);
\draw[e] (20)--(21);
\draw[e] (11)--(21);
\draw[e] (11)--(01);
\draw[e] (40)--(50);
\draw[e] (40)--(30);
\draw[e] (31)--(30);
\draw[e] (31)--(41);
\draw[e] (41)--(51);
\draw[e] (31)--(21);
\begin{scope}[xshift=5cm]
\node[v] (20) at (2,0) {};
\node[v] (30) at (3,0) {};
\node[v] (11) at (1,1) {};
\node[v] (21) at (2,1) {};
\node[v] (31) at (3,1) {};
\node[v] (41) at (4,1) {};
\draw[e] (20)--(21);
\draw[e] (21)--(11);
\draw[e] (30)--(31);
\draw[e] (31)--(41);
\draw[e] (31)--(21);
\end{scope}
\end{tikzpicture}
$,
then
$
\vec F_\rc = 
\begin{tikzpicture}[scale=.5]
\foreach \x in {0,...,5}
	\foreach \y in {0,1}
		\node[v] (\x\y) at (\x,\y) {};
\draw[e] (00)--(10);
\draw[e] (10)--(20);
\draw[e] (11)--(21);
\draw[e] (11)--(01);
\draw[e] (40)--(50);
\draw[e] (40)--(30);
\draw[e] (31)--(41);
\draw[e] (41)--(51);
\begin{scope}[xshift=5cm]
\node[v] (20) at (2,0) {};
\node[v] (30) at (3,0) {};
\node[v] (11) at (1,1) {};
\node[v] (21) at (2,1) {};
\node[v] (31) at (3,1) {};
\node[v] (41) at (4,1) {};
\draw[e] (20)--(21);
\draw[e] (21)--(11);
\draw[e] (30)--(31);
\draw[e] (31)--(41);
\end{scope}
\end{tikzpicture}
$.
\end{example}

\begin{lemma} \label{lem:unique-cut}
	Let $F$ be a subgraph of a tree $T$ with $V(F) = V(T)$ such that all components of $F$ have equal number of vertices. 
	Then no other subgraph of $T$ is isomorphic to $F$, i.e., $|\Sub_F(T)| = 1$.
\end{lemma}

\begin{proof}
	Suppose $F' \ne F$ is another subgraph of $T$ and $F' \cong F$. Let $C$ be a component of $F$, $C'$ a component of $F'$, with $V(C)$ and $V(C')$ overlapping but not identical, so that there exists $uv \in E(C')$ with $u \in V(C)$ and $v \notin V(C)$. Let $(B,v)$ be the branch of $T$ cut from $uv$. Then $V(B)$ is a union of components of $F$, so $|B|$ is divisible by $|C|$. On the other hand, $V(B)$ partitions into $V(C') \cap V(B)$ together with components of $F'$, so $V(B)$ is not divisible by $|C|$. Contradiction.
\end{proof}

Now we are ready to prove the main result of this section, that the undirected structure of an impartial tree is recursively bridge-mirrored.

\begin{proof}[Proof of \cref{prop:undirected}]
Let $\vec H$ be a weakly connected impartial digraph. By \cref{prop:tree}, $H$ is a tree. All components of $H_\rc$ are isomorphic due to the recursive cutting procedure. 
If $H_\rc$ has at least one edge, then \cref{prop:sgn-sum} with $\vec F = \vec H_\rc$ yields a contradiction due to \cref{lem:unique-cut}. Thus $H_\rc$ is edgeless, and hence $H$ is recursively bridge-mirrored, as desired. 
\end{proof}

\section{Directed structure} \label{sec:directed}

Recall from \cref{prop:components} that a digraph is impartial if and only if all its weakly connected components are impartial. 
From \cref{prop:undirected} we know that the undirected structure of a weakly connected impartial tree is a recursively bridge-mirrored tree. The goal of this section is to complete the proof of our main result \cref{thm:main}, which follows from the next claim  showing that the edge-orientation of such a tree is compatible with the recursive bridge-mirroring. 

\begin{proposition} \label{prop:directed}
Every weakly connected impartial digraph is recursively bridge-mirrored.	
\end{proposition}

In fact, it suffices to show that the directed structure is compatible with the involution of the tree.

\begin{lemma} \label{lem:top-directed}
	Let $\vec H$ be a weakly connected impartial digraph with $|H| > 1$. Then the odd automorphism of $H$ preserves the orientations of all edges of $\vec H$ except for the mirror-bridge.
\end{lemma}

\begin{remark}
	An easy induction argument shows that every undirected recursively bridge-mirrored tree with at least one edge has exactly two automorphisms: the identity map, and the odd automorphism that swaps the two half-branches and reverses the mirror-bridge.
\end{remark}

\begin{proof}[Proof of \cref{prop:directed} assuming \cref{lem:top-directed}]
We apply induction on $|H|$. There is nothing to show if $|H| = 1$. 
Now assume $|H| > 1$, so that \cref{lem:top-directed} applies. Let $\vec H_1$ denote $\vec H$ after removing its mirror-bridge. 
Then \cref{lem:top-directed} implies that every copy of $\vec H_1$ in a tournament can be extended uniquely to a copy of $\vec H$, and hence $\vec H_1$ must be impartial.
By \cref{prop:components}, each component of $\vec H_1$ is impartial, and thus recursively bridge-mirrored by induction. By \cref{lem:top-directed} again, we see that $\vec H$ is recursively bridge-mirrored as well.
\end{proof}

In the rest of this section, we prove \cref{lem:top-directed}. By \cref{prop:undirected}, $H$ is a recursively bridge-mirrored tree. Let $\vec T$ be a recursively bridge-mirrored digraph with the same undirected structure $H = T$. Let $\tau$ denote the odd automorphism of $H$.

For each undirected edge $e \in E(H) = E(T)$, let $z_e = 1$ if the orientations of $e$ on $\vec H$ and $\vec T$ agree, and $z_e = -1$ otherwise. To prove \cref{lem:top-directed}, it remains to show that $z_e = z_{\tau e}$ for all $e \in E(H)$.

By \cref{prop:sgn-sum}, if $\vec F$ is an even digraph with $|H|$ vertices and at least one edge, then $|\Sub_F(H)|$ is even and
\[
\prod_{\vec F' \in \Sub_F(\vec H)} \sgn(\vec F, \vec F') = \prod_{\vec F' \in \Sub_F(\vec T)} \sgn(\vec F, \vec F') = (-1)^{|\Sub_F(H)|/2}.
\]
Multiplying the two products together, we obtain that for every even graph $F$ with $|H|$ vertices and at least one edge,
\begin{equation}\label{eq:T-E}
\prod_{F' \in \Sub_F(H)} \prod_{e \in E(F')} z_e = 1.
\end{equation}
We shall use \cref{eq:T-E} with $F$ being
\[
T_f := (T \setminus f)_\rc
\]
for every edge $f$ of $T$ other than its mirror-bridge. Note that $T_f$ has at least one edge as long as $f$ is not the mirror-bridge of $T$, since the recursive cutting procedure preserves the largest odd factor of the number of vertices of each component.

We are interested in which $z_e$'s appear an odd number of times on the left-hand side of \cref{eq:T-E}. Let
\begin{align*}
S_F &:= \{e \in E(T) : e \text{ is contained in an odd number of subgraphs of $H$ isomorphic to } F \} 
\\ &= \{e \in E(T) : |\Sub_F(T)| - |\Sub_F(T \setminus e)| \equiv 1 \pmod 2\}.
\end{align*}
Since $z_e = \pm 1$, \cref{eq:T-E} implies that for all even $F$ with at least one edge
\begin{equation}\label{eq:system-Sf}
\prod_{e \in S_F} z_e = 1.
\end{equation}

\begin{example} \label{example:path-directed}
We work out the above computation explicitly for a small example.
Let $T$ be a path with 8 vertices. 
\[
T = 
\begin{tikzpicture}[baseline=-.7ex]
\foreach \x in {1,...,8}
	\node[v,label=above:{$\x$}] (\x) at (\x,0) {};
\draw (1)--(2);
\draw (2)--(3);
\draw (3)--(4);
\draw (4)--(5);
\draw (5)--(6);
\draw (6)--(7);
\draw (7)--(8);
\end{tikzpicture}
\phantom{= T}
\]
We have
\[
T_{12} =
\begin{tikzpicture}
\foreach \x in {1,...,8}
	\node[v] (\x) at (\x,0) {};
\draw (1)(2)--(3)--(4)--(5)--(6)--(7)--(8);
\end{tikzpicture}
\phantom{= T_{12}}
\]
Thus $\Sub_{T_{12}}(H)$ has 2 elements:
\[
\begin{tikzpicture}
\begin{scope}
\foreach \x in {1,...,8}
	\node[v] (\x) at (\x,0) {};
\draw (1)(2)--(3)--(4)--(5)--(6)--(7)--(8);
\end{scope}
\begin{scope}[yshift=-.5cm]
\foreach \x in {1,...,8}
	\node[v] (\x) at (\x,0) {};
\draw (1)--(2)--(3)--(4)--(5)--(6)--(7)(8);
\end{scope}
\end{tikzpicture}
\]
So $S_{T_{12}} = \{12, 78\}$. Thus, by \cref{eq:system-Sf},
\[
z_{12} = z_{78}.
\]
Likewise, $\Sub_{T_{34}}(H)$ has 2 elements:
\[
\begin{tikzpicture}
\begin{scope}
\foreach \x in {1,...,8}
	\node[v] (\x) at (\x,0) {};
\draw (1)--(2)--(3)(4)--(5)--(6)--(7)--(8);
\end{scope}
\begin{scope}[yshift=-.5cm]
\foreach \x in {1,...,8}
	\node[v] (\x) at (\x,0) {};
\draw (1)--(2)--(3)--(4)--(5)(6)--(7)--(8);
\end{scope}
\end{tikzpicture}
\]
So $S_{T_{34}} = \{34,56\}$. Thus, by  \cref{eq:system-Sf},
\[
z_{34} = z_{56}.
\]
Finally, we have
\[
T_{23} = 
\begin{tikzpicture}
\foreach \x in {1,...,8}
	\node[v] (\x) at (\x,0) {};
\draw (1)(2)(3)--(4)--(5)(6)--(7)--(8);
\end{tikzpicture}
\phantom{= T_{12}}
\]
So $\Sub_{T_{23}}(H)$ has 6 elements:
\[
\begin{tikzpicture}
\begin{scope}
\foreach \x in {1,...,8}
	\node[v] (\x) at (\x,0) {};
\draw (1)(2)(3)--(4)--(5)(6)--(7)--(8);
\end{scope}
\begin{scope}[yshift=-.5cm]
\foreach \x in {1,...,8}
	\node[v] (\x) at (\x,0) {};
\draw (1)(2)--(3)--(4)(5)(6)--(7)--(8);
\end{scope}
\begin{scope}[yshift=-1cm]
\foreach \x in {1,...,8}
	\node[v] (\x) at (\x,0) {};
\draw (1)--(2)--(3)(4)(5)(6)--(7)--(8);
\end{scope}
\begin{scope}[yshift=-1.5cm]
\foreach \x in {1,...,8}
	\node[v] (\x) at (\x,0) {};
\draw (1)(2)--(3)--(4)(5)--(6)--(7)(8);
\end{scope}
\begin{scope}[yshift=-2cm]
\foreach \x in {1,...,8}
	\node[v] (\x) at (\x,0) {};
\draw (1)--(2)--(3)(4)(5)--(6)--(7)(8);
\end{scope}
\begin{scope}[yshift=-2.5cm]
\foreach \x in {1,...,8}
	\node[v] (\x) at (\x,0) {};
\draw (1)--(2)--(3)(4)--(5)--(6)(7)(8);
\end{scope}
\end{tikzpicture}
\]
Thus $S_{T_{23}} = \{12,23,34,56,67,78\}$. Thus by \cref{eq:system-Sf},
\[
z_{12}z_{23}z_{34} = z_{56}z_{67}z_{78},
\]
and we deduce $z_{23} = z_{67}$ as well.  \qed 
\end{example}

The next lemma shows that \cref{eq:system-Sf} can be made into an upper-triangular linear system (mod 2) with an appropriate choice and ordering of subgraphs $F$'s.

\begin{lemma}\label{lem:upper-triangular}
	Let $T$ be a recursively bridge-mirrored tree. For every edge $f$ of $T$ other than its mirror-bridge, we have $f, \tau f \in S_{T_f}$ and $|E(T_e)| > |E(T_f)|$ for all $e \in S_{T_f} \setminus \{f, \tau f\}$.
\end{lemma}

\begin{proof}[Proof of \cref{lem:top-directed} assuming \cref{lem:upper-triangular}]
	We see that $S_F$ is invariant under the automorphism $\tau$ of $T$ for every $F \subseteq T$. Let $T'$ be a half-branch of $T$. Write $w_e = z_e z_{\tau e}$ for each $e \in E(T')$. Then \cref{eq:system-Sf} gives
	\[
	\prod_{e \in S_{T_f} \cap E(T')} w_e = 1 \quad \text{ for all }  f \in E(T').
	\]
	Then by \cref{lem:upper-triangular}, we can solve this system of equations to yield $w_e = 1$ for all $e \in E(T')$ by decreasing induction on $|E_{T_e}|$ (break ties arbitrarily). Alternatively, this system is equivalent to a linear system of equations (mod 2) whose coefficient matrix is upper-triangular with 1's on the diagonal when the variables $w_e$ are sorted according to $|E(T_e)|$ (breaking ties arbitrarily). 
	
	Thus, for all $e \in E(T')$, $z_e z_{\tau e} = w_e = 1$, and hence $\tau$ preserves the orientation of $e$ in $\vec H$ since $\tau$ does the same for $\vec T$.
\end{proof}

We shall prove \cref{lem:upper-triangular} by pairing up copies of $F$ in $T$ by applying an odd automorphism on some subgraph of $T$.

\begin{lemma} \label{lem:sub-rc}
Let $F$ be an even subgraph of a forest $G$. Then
\[
|\Sub_F(G)| \equiv |\Sub_F(G_\rc)| \pmod 2.
\]
\end{lemma}

\begin{proof}
Order the edges of $G$ according to the order that they are removed when $G$ is recursively cut, breaking ties arbitrarily, i.e., edges of $G_i \setminus G_{i+1}$ appear earlier than those of $G_{i+1}\setminus G_{i+2}$, with the edges of $G_\rc$ appearing last.

Let $F'$ be a copy of $F$ in $G$ but not contained in $G_\rc$. Let $e$ be the first edge of $G$ in the above order contained in $F'$, and $k$ the largest integer such that $F' \subseteq G_k$.
Then $e$ is a mirror-bridge of its component in $G_k$. Let $\sigma$ denote the automorphism of $G_k$ obtained by applying an odd automorphism on the component of $e$ and leaving other components of $G_k$ fixed. Note that $\sigma$ does not induce an automorphism of $F'$, or else the component of $F'$ containing $e$ would be odd, contradicting $F$ being even. 

The same procedure applied to $\sigma(F')$ recovers the same $e$, and hence recovers the original $F'$ (as long as we pick a consistent choice of an odd automorphism for each connected odd subgraph). Thus we have paired up copies of $F$ in $G$ but not contained in $G_\rc$, thereby proving the lemma.
\end{proof}

\begin{example}
	Let us illustrate the pairing in the above proof. Let $G$ be the following graph, with the first three edges in the recursive cutting removal ordering labeled.
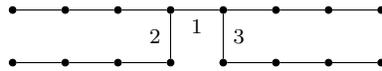
\[
\begin{tikzpicture}
	\foreach \x in {0,...,7}
	\foreach \y in {0,1}
	\node[v] (\x\y) at (\x,\y) {};
\draw (00)--(10)--(20)--(30)--node[left, font=\scriptsize] {$2$}(31)--(21)--(11)--(01);
\draw (70)--(60)--(50)--(40)--node[right, font=\scriptsize] {$3$}(41)--(51)--(61)--(71);
\draw (31)-- node[below, font=\scriptsize] {$1$}(41);
\end{tikzpicture}
\]
Then the following two highlighted subgraphs are paired together in the proof via reflecting the right half-branch of $G$ across edge 3 (highlighted) in 
$G_1 = 
\begin{tikzpicture}[scale=.5]
	\foreach \x in {0,...,7}
	\foreach \y in {0,1}
	\node[v] (\x\y) at (\x,\y) {};
\draw (00)--(10)--(20)--(30)--(31)--(21)--(11)--(01);
\draw (70)--(60)--(50)--(40) (41)--(51)--(61)--(71);
\draw[b] (40)--(41);
\end{tikzpicture}$:

\[
\begin{tikzpicture}
	\foreach \x in {0,...,7}
	\foreach \y in {0,1}
	\node[v] (\x\y) at (\x,\y) {};
	\draw (00)--(10)--(20)--(30)--(31)--(21)--(11)--(01);
	\draw (70)--(60)--(50)--(40)--(41)--(51)--(61)--(71);
	\draw (31)-- (41);
	\draw[b] (71)--(61)--(51) (41)--(40)--(50) (30)--(20)--(10);
\end{tikzpicture}
\hspace{4em}
\begin{tikzpicture}
	\foreach \x in {0,...,7}
	\foreach \y in {0,1}
	\node[v] (\x\y) at (\x,\y) {};
\draw (00)--(10)--(20)--(30)--(31)--(21)--(11)--(01);
\draw (70)--(60)--(50)--(40)--(41)--(51)--(61)--(71);
\draw (31)-- (41);
	\draw[r] (70)--(60)--(50) (40)--(41)--(51) (30)--(20)--(10);
\end{tikzpicture}
\]
\end{example}

\begin{proof}[Proof of \cref{lem:upper-triangular}]
	By \cref{lem:sub-rc}, $|\Sub_{T_f}(T)| \equiv |\Sub_{T_f}(T_\rc)| = 0 \pmod 2$ since $T_\rc$ is edgeless and $T_f$ is even and not edgeless. Also by \cref{lem:sub-rc} (recall $T_e = (T\setminus e)_\rc$),
	\begin{align*}
	|\Sub_{T_f}(T\setminus e)|
	&\equiv |\Sub_{T_f}(T_e)| \pmod 2
	\\&= \begin{cases}
 	 1 & \text{if } e \in \{f, \tau f\}, \\
 	 0 & \text{if } |E(T_e)| \le |E(T_f)| \text{ and } e \notin \{f,\tau f\},
 	  \end{cases}
	\end{align*}
	(we do not say what happens when $|E(T_e)| > |E(T_f)|$)
	since, in the first case, $T_f \cong T_e$, and, in the second case, either $|E(T_e)| < |E(T_f)|$ so that $T_e$ cannot contain $T_f$ as a subgraph, or $|E(T_e)| = |E(T_f)|$, in which case $|\Sub_{T_f}(T_e)| = 0$ unless $T_e \cong T_f$, which is ruled out by the upcoming \cref{prop:tree-rc} in the next section.
\end{proof}

\section{Properties of undirected recursively bridge-mirrored trees} \label{sec:property}

It remains to prove the following claim, which was invoked at the end of the previous section. Recall that $T_e := (T\setminus e)_\rc$.

\begin{proposition} \label{prop:tree-rc}
	Let $T$ be a recursively bridge-mirrored tree. Let $e_1$ and $e_2$ be two distinct edges in $T$. If $T_{e_1} \cong T_{e_2}$, then the odd automorphism of $T$ carries $e_1$ to $e_2$.
\end{proposition}

\begin{lemma} \label{lem:power-of-2}
Let $T$ be a recursively bridge-mirrored tree with $2^k$ vertices whose recursive cutting results in $(T_0 = T, T_1, \dots, T_k = T_\rc)$. Then for each $1 \le i \le k$, $2^{k-i}$ is the largest power of $2$ that divides the number of vertices in a branch of $T$ cut from an edge in $T_{i-1}\setminus T_i$.
\end{lemma}

\begin{proof} Induction on $i$ (easy and omitted).
\end{proof}

\begin{lemma} \label{lem:adjacent-edge-sizes}
Let $T$ be a recursively bridge-mirrored tree. If two distinct branches of $T$ have the same number of vertices, then they are cut from edges that do not share endpoints.
\end{lemma}

\begin{proof}
By \cref{lem:power-of-2}, the two cut edges lie in the same  $T_i \setminus T_{i+1}$ for some $i$, and no two edges of $T_i \setminus T_{i+1}$ share a vertex since they are the mirror-bridges of the components of $T_i$.
\end{proof}

\begin{lemma}[Minority branch] \label{lem:minority-branch}
Let $T$ be a recursively bridge-mirrored tree. Let $(A_1, u_1)$ and $(A_2, u_2)$ be two branches of $T$. 
\begin{enumerate}
	\item[(a)] If $|A_1| = |A_2| \le |T|/2$, then $(A_1, u_1)$ and $(A_2, u_2)$ are isomorphic as rooted trees.
	\item[(b)] If $|T|/4 \leq |A_1| = |A_2| \le |T|/2$, then $T$ has an automorphism carrying $(A_1, u_1)$ to $(A_2, u_2)$.
\end{enumerate}
\end{lemma}

\begin{remark}
Here is an example of two branches (highlighted) where $(A_1, u_1) \cong (A_2, u_2)$ but no automorphism of $T$ carries one to the other:
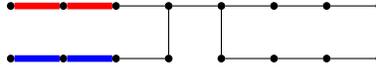
\[
\begin{tikzpicture}
\foreach \x in {1,...,8}
	\foreach \y in {0,1}
		\node[v] (\x\y) at (\x,\y) {};
\draw[b] (10)--(20)--(30);
\draw[r] (11)--(21)--(31);
\draw (30)--(40)--(41)--(31) (41)--(51)--(61)--(71)--(81) (51)--(50)--(60)--(70)--(80);
\end{tikzpicture}
\]
\end{remark}

\begin{proof}
	We apply induction on $|T|$, with $|T| \le 2$ being trivial. Let $T'$ be a half-branch of $T$. Applying an automorphism of $T$ if necessary, we may assume that $A_1, A_2 \subseteq T'$.
	
	If $|A_1| < |T|/4$, then (a) follows by applying the induction hypothesis to $|T'|$.
	
	Now assume $|T|/4 \le |A_1| \le |T|/2$. Then both $A_1$ and $A_2$ have at least $|T'|/2$ vertices and hence each contains at least one vertex of the mirror-bridge of $T'$. So $u_1$ and $u_2$ both lie on the path between the mirror-bridges of $T$ and $T'$. It follows that $u_1 = u_2$ or else one of $A_1$ and $A_2$ would strictly contain the other, which is impossible as $|A_1| = |A_2|$. Hence $(A_1, u_1) = (A_2, u_2)$, and the result follows recalling the automorphism we may have applied initially.
\end{proof}

Let $\min(G)$ denote the number of vertices in the smallest component of $G$. For $F \subseteq G$, write $m(G, F)$ for the multiset with elements $\min(G \setminus e)$ over all $e \in E(F)$. For graphs $A$ and $B$, we write $A \cap B$ for the graph $(V(A) \cap V(B), E(A) \cap E(B))$.

\begin{example}
In the following graph $G$, each edge $e$ is labeled by $\min(G \setminus e)$:
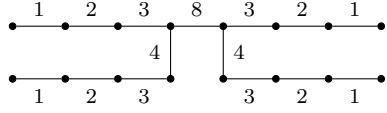
\[
\begin{tikzpicture}[font=\scriptsize]
\foreach \x in {0,...,7}\foreach \y in {0,1}
\node[v] (\x\y) at (\x,\y) {};
\draw (00)--node[below] {$1$} (10)
          --node[below] {$2$} (20)
          --node[below] {$3$} (30)
          --node[left] {$4$} (31)
          --node[above] {$3$} (21)
          --node[above] {$2$} (11)
          --node[above] {$1$} (01);
\draw (70)--node[below] {$1$} (60)
          --node[below] {$2$} (50)
          --node[below] {$3$} (40)
          --node[right] {$4$} (41)
          --node[above] {$3$} (51)
          --node[above] {$2$} (61)
          --node[above] {$1$} (71);
\draw (31)-- node[above] {$8$}(41);
	
\end{tikzpicture}
\]
\end{example}

\begin{lemma} \label{lem:majority-1}
Let $T$ be a recursively bridge-mirrored tree and $T'$ a half-branch of $T$. Let $(B_1, v_1)$ and $(B_2, v_2)$ be branches of $T$ with both $v_1$ and $v_2$ contained in $T'$. If $B_1$ and $B_2$ are isomorphic as trees and $|B_1| = |B_2| > |T|/2$, then 
$m(B_1, B_1 \cap T') = m(B_2, B_2 \cap T')$ and 
$m(T, B_1 \cap T') = m(T, B_2 \cap T')$.
\end{lemma}

\begin{proof}
For every edge $e \in E(T)$ not contained in $T'$, one has $\min(B_1 \setminus e) = \min(B_2 \setminus e)$ since one of the components of $B_1 \setminus e$ and $B_2 \setminus e$ coincide, namely the component disjoint from $T'$. Removing from $m(B_1, B_1) = m(B_2, B_2)$ (as $B_1 \cong B_2$) the above contributions (namely the element $\min(B_1 \setminus e) = \min(B_2 \setminus e)$ for each $e \in E(T)$ not in $T'$), we obtain $m(B_1, B_1 \cap T') = m(B_2, B_2 \cap T')$.

For each $i = 1, 2$, let $(A_i, u_i)$ be the branch of $T$ obtained by removing $B_i$ (so that $|A_i| + |B_i| = |T|$). 
For any edge $e$ in $A_i$, the smaller branch of $T$ cut from $e$ coincides with the branch of $A_i$ cut from $e$ not containing $u_i$. By \cref{lem:minority-branch}, $(A_1, u_1) \cong (A_2, u_2)$, and this isomorphism induces the equality of multisets $m(T, A_1) = m(T, A_2)$. Also $\min(T \setminus u_iv_i) = |A_i|$, which is the same for $i=1,2$. Removing from $m(T, T')$ the above contributions (namely $m(T, A_1) = m(T, A_2)$ along with the element $|A_1| = |A_2|$), we find that $m(T, B_1 \cap T') = m(T, B_2 \cap T')$ as claimed.	
\end{proof}

\begin{lemma} \label{lem:majority-2}
Let $T$ be a recursively bridge-mirrored tree with mirror-bridge $f$, and $T'$ a half-branch of $T$. Let $A$ and $B$ be the two branches of $T$ cut from $g \in E(T')$, with $A$ contained in $T'$. Then for every $e \in E(B \cap T')$,
\[
	\min(T \setminus e) - \min(B \setminus e) = 
		\begin{cases}
			|A| & \text{if $e$ lies on the path in $T$ connecting $f$ and $g$,} \\
			0 & \text{otherwise.}
		\end{cases}
	\]
\end{lemma}

\begin{proof}
	If $e$ lies on the path in $T$ connecting $f$ and $g$, then the larger branches of $T$ and $B$ cut from $e$ coincide, and hence the vertices of the smaller branches differ by $V(A)$. If $e$ does not lie on this path, then the smaller branches of $T$ and $B$ cut from $e$ coincide.
\end{proof}

\begin{lemma}[Majority branch] \label{lem:majority-branch}
Let $T$ be a recursively bridge-mirrored tree. Let $(B_1, v_1)$ and $(B_2, v_2)$ be two branches of $T$ such that $B_1$ and $B_2$ are isomorphic and each has at least $|T|/2$ vertices. Then $T$ has an automorphism carrying $v_1$ to $v_2$.
\end{lemma}

\begin{proof}
If $|B_1| = |B_2| = |T|/2$, then $B_1$ and $B_2$ are both half-branches of $T$. So assume that $|B_1| = |B_2| > |T|/2$. 

Suppose that each $B_i$ is a branch of $T$ cut from edge $f_i$. We may assume without loss of generality that $f_1$ and $f_2$ lie in the same half-branch $T'$ of $T$. We wish to show that $f_1 = f_2$.

By \cref{lem:majority-1}, $m(B_1, B_1 \cap T') = m(B_2, B_2 \cap T')$ and 
$m(T, B_1 \cap T') = m(T, B_2 \cap T')$.
By \cref{lem:majority-2}, for each $i = 1, 2$, $m(B_i, B_i \cap T')$ can be obtained from $m(T, B_i \cap T')$ by subtracting $|T| - |B_i|$ from an element equal to $\min (T\setminus e) \in m(T, B_i \cap T')$ for each edge $e$ on the path $P_i$ from  the mirror-bridge of $T$ to $f_i$. 
It follows that the multiset of elements $\min(T \setminus e)$, as $e$ ranges over the edges of $P_i$, does not depend on $i$. 

For each $i$, as the edge $e$ walks along $P_i$ from the mirror-bridge of $T$ to $f_i$, 
$\min(T\setminus e)$ is strictly decreasing. 
Also, by \cref{lem:adjacent-edge-sizes}, the size of the smaller component of $T\setminus e$ must differ at the first instance when $P_1$ and $P_2$ diverge, 
contradicting the claim at the end of the previous paragraph. 
Therefore $f_1 = f_2$.
\end{proof}

\begin{lemma}\label{lem:majority-odd}
Let $T$ be a recursively bridge-mirrored tree with at least two vertices. Let $B_1$ and $B_2$ be two branches of $T$ with $|B_1| = |B_2| \ge |T|/2$. If both $B_1$ and $B_2$ are odd, then $T$ has an automorphism carrying $B_1$ to $B_2$.	
\end{lemma}

\begin{proof}
For each $i = 1, 2$, since $B_i$ is odd, it has a mirror-bridge $e_i$ by \cref{lem:mirror-bridge}.
Let $(C_i, r_i)$ be the half-branch of $B_i$ not containing the mirror-bridge of $T$.
Since $|T|/4 \le |C_1| = |C_2| \le |T|/2$, by \cref{lem:minority-branch}, $T$ has an automorphism carrying $(C_1, r_1)$ to $(C_2, r_2)$.
Since $(C_1, r_1) \cong (C_2, r_2)$, we have $B_1 \cong B_2$, and the conclusion follows by \cref{lem:majority-branch}.
\end{proof}

Now we are ready to prove the main result of this section, thereby completing the proof of the main result \cref{thm:main} of this paper.

\begin{proof}[Proof of \cref{prop:tree-rc}]
	For each $i=1,2$, let $A_i$ and $B_i$ be the two components of $T\setminus e_i$, with $|A_i| \le |B_i|$.
	
	If $B_1$ is even, then it is a component of $T_{e_1}$ with at least half of the vertices. Since $T_{e_1} \cong T_{e_2}$, $B_2$ must also be an even component of $T_{e_2}$, and $B_1 \cong B_2$, and the conclusion follows from \cref{lem:majority-branch}.
	
	So assume that both $B_1$ and $B_2$ are odd. Since $T_{e_1} \cong T_{e_2}$, we have $|B_1| = |B_2|$ (note we can identify which components of $T_{e_1}$ came from $B_1$ versus $A_1$ based on the largest odd factor of its order), and the conclusion then follows from \cref{lem:majority-odd}.
\end{proof}

\section{Further questions} \label{sec:further}

\subsection{Positive and negative digraphs}
For undirected graphs, Sidorenko's conjecture says that for all bipartite $H$, one has $t(H, W) \ge t(K_2, W)^{|E(H)|}$ for every graphon $W \colon [0,1]^2 \to [0,1]$, where graphons satisfy $W(x,y) = W(y,x)$. As mentioned in the introduction, Sidorenko's conjecture has been proved for several families of $H$ but remains open in general.

\cref{prop:tourneyon} tells us that $\vec H$ is impartial if and only if $t(\vec H, W) = 2^{-|E(H)|}$ for all tourneyons $W$, which satisfy $W(x,y) + W(y,x) = 1$ for all $x,y \in [0,1]^2$, unlike graphons. Our classification of impartial graphs can be viewed as the equality case of a directed analog of Sidorenko's conjecture. 

What about inequalities for digraphs? 

Let us call a digraph $\vec H$ \emph{positive} if $t(\vec H, W) \ge 2^{-|E(H)|}$ for all tourneyons $W$, and \emph{negative} if $t(\vec H, W) \le 2^{-|E(H)|}$ for all tourneyons $W$.

\newcommand{\pathaa}{\begin{tikzpicture}[scale=.5]
\node[v] (0) at (0,0) {};
\node[v] (1) at (1,0) {};
\node[v] (2) at (2,0) {};
\draw[e] (0)--(1);
\draw[e] (1)--(2);
\end{tikzpicture}
}
\newcommand{\pathba}{\begin{tikzpicture}[scale=.5]
\node[v] (0) at (0,0) {};
\node[v] (1) at (1,0) {};
\node[v] (2) at (2,0) {};
\draw[e] (1)--(0);
\draw[e] (1)--(2);
\end{tikzpicture}
}
\newcommand{\pathab}{\begin{tikzpicture}[scale=.5]
\node[v] (0) at (0,0) {};
\node[v] (1) at (1,0) {};
\node[v] (2) at (2,0) {};
\draw[e] (0)--(1);
\draw[e] (2)--(1);
\end{tikzpicture}
}
\newcommand{\pathaaa}{\begin{tikzpicture}[scale=.5]
\node[v] (0) at (0,0) {};
\node[v] (1) at (1,0) {};
\node[v] (2) at (2,0) {};
\node[v] (3) at (3,0) {};
\draw[e] (0)--(1);
\draw[e] (1)--(2);
\draw[e] (2)--(3);
\end{tikzpicture}
}
\newcommand{\pathaba}{\begin{tikzpicture}[scale=.5]
\node[v] (0) at (0,0) {};
\node[v] (1) at (1,0) {};
\node[v] (2) at (2,0) {};
\node[v] (3) at (3,0) {};
\draw[e] (0)--(1);
\draw[e] (2)--(1);
\draw[e] (2)--(3);
\end{tikzpicture}
}
\newcommand{\pathaab}{\begin{tikzpicture}[scale=.5]
\node[v] (0) at (0,0) {};
\node[v] (1) at (1,0) {};
\node[v] (2) at (2,0) {};
\node[v] (3) at (3,0) {};
\draw[e] (0)--(1);
\draw[e] (1)--(2);
\draw[e] (3)--(2);
\end{tikzpicture}
}
\newcommand{\pathbaa}{\begin{tikzpicture}[scale=.5]
\node[v] (0) at (0,0) {};
\node[v] (1) at (1,0) {};
\node[v] (2) at (2,0) {};
\node[v] (3) at (3,0) {};
\draw[e] (1)--(0);
\draw[e] (1)--(2);
\draw[e] (2)--(3);
\end{tikzpicture}
}

\begin{example}[2-edge path]
The digraphs \pathab and \pathba are positive since for any tourneyon $W$,
\[
\int W(x,y)W(z,y) \,dxdy
= \int \paren{\int W(x,y) \, dx}^2 \, dy
\ge \paren{\int W(x,y) \, dxdy}^2 
= 2^{-2}
\]
whereas \pathaa is negative since $2t(\pathaa, W) + t(\pathab, W) + t(\pathba, W) = 1$ for any tourneyon $W$.
\end{example}

Let $H$ be a bipartite graph with bipartition $A \cup B$. Orient all edges from $A$ to $B$ to obtain a digraph $\vec H$. If $H$ satisfies a strengthened bipartite version of Sidorenko's conjecture (all known cases of Sidorenko's conjecture are proved under this strengthening), then $\vec H$ is positive.
 
\begin{example}[3-edge path]
The digraph \pathaba is positive, as it follows from the proof of Sidorenko's conjecture for a 3-edge path (see the responses to this MathOverflow post~\cite{overflow-sid} for some nice and short proofs).

The digraph \pathaaa is negative since $t(\pathaba, W) + t(\pathaaa, W) = 1/4$ for all tourneyons $W$.

The remaining orientations $\pathaab$ and $\pathbaa$ are impartial.
\end{example}

There are recursive ways to build up positive graphs. For example, construct $\vec H$ by taking a disjoint union of $\vec H_1$ and $\vec H_2$ and adding a new directed edge from every vertex of $\vec H_1$ to every vertex of $\vec H_2$. We leave it as a fun exercise\footnote{Hint: by first fixing the embedding of $\vec H_1$, and applying the positivity of $\vec H_2$, we can replace $\vec H_2$ by an empty graph on $|H_2|$ vertices. Likewise with $\vec H_1$. Then it remains to show that the complete bipartite graph is positive, which follows from H\"older's inequality.} to show that if $\vec H_1$ and $\vec H_2$ are positive, then so is $\vec H$.

As a corollary of the above recursive construction, we see that all transitive tournaments are positive, which was known~\cite{CR17}. Also, a positive tournament must be transitive since it has to be embedded into a transitive tournament.

It appears to be a challenging problem to classify all positive and negative digraphs. In particular, the problem of classifying positive digraphs includes Sidorenko's conjecture as a special case.

\begin{problem} \label{prob:pos-neg}
	Determine all positive and negative digraphs.
\end{problem}

The problem is somewhat reminiscent of the \emph{positive graph conjecture} \cite{ACHL16}. We say that a graph $G$ is \emph{positive} if $t(H, W)\ge 0$ for every measurable $W \colon [0,1]^2 \to [-1,1]$ 
(i.e., graphons but allowing negative values) with $W(x,y) = W(y,x)$. For example, an application of the Cauchy--Schwarz inequality shows that $H = C_4$ is positive. A similar Cauchy--Schwarz application shows that any graph $H$ obtained by gluing some graph to itself along an independent set is positive, and the conjecture says that every positive graph has this form.

Recently, a reverse Sidorenko inequality was established~\cite{SSSZ}, showing that, for instance, for a triangle-free $d$-regular graph $H$, the $H$-density in a graph is always upper-bounded by the appropriately normalized $K_{d,d}$-density. It may be interesting to explore directed versions of such reverse Sidorenko inequalities.

\subsection{Directed hypergraphs}

The notion of impartiality can also be generalized to hypergraphs. Here one needs to specify what is meant by a directed hypergraph and what is the analog of a tournament.  What's a ``directed edge''?

Here are two possible definitions for how to orient a triple in a 3-uniform hypergraph.

One possible definition of a directed triple is to pick one of the 6 permutations of the triple.  Then the corresponding generalization of a tourneyon is a measurable function $W \colon [0,1]^3 \to [0,1]$ with $W(x,y,z) + W(x,z,y) + W(y,x,z) + W(y,z,x) + W(z,x,y) + W(z,y,x) = 1$.

Another possible definition of a directed triple is to pick one of the two possible signs for a permutation of a triple (i.e., pick one of the two equivalence classes of the 6 permutations). Then, on top of the earlier constraint for $W$, one should also add $W(x,y,z) = W(y,z,x) = W(z,x,y)$ and $W(x,z,y) = W(z,y,x) = W(y,x,z)$.

For either notion, we can define an \emph{impartial} directed 3-uniform hypergraph $\vec H$ as one whose density $t(\vec H, W)$
is constant for all 3-uniform hypertourneyons $W$. Similarly, we say that $\vec H$ is \emph{positive}/\emph{negative} if $t(\vec H, W)$ is \emph{minimized}/\emph{maximized} by the constant tourneyon.

For $k$-uniform hypergraphs, we generalize the above discussion by selecting a proper normal subgroup $N$ of the symmetric group $S_k$, and ``orienting'' each $k$-tuple by selecting a coset of $N$ acting on the $k$-tuple. The corresponding constraint for hypertourneyons is $W(x_1, \dots, x_k) = W(x_{\sigma(1)}, \dots, x_{\sigma(k)})$ for all $\sigma \in N$. Recall that the only normal subgroups of $S_k$ with $k \ge 2$ are the trivial group and the alternating group, except $k=4$ where there is also the Klein four group. In the $k=3$ discussion earlier, the first possibility corresponds to $N$ being the trivial group, and the second corresponds to $N$ being the alternating group.

\begin{problem}
	Determine all impartial directed hypergraphs for each notion of directed hypergraphs.
\end{problem}

We can also ask for positive and negative directed hypergraphs. Though, we do not even know which undirected hypergraphs satisfy the generalization of Sidorenko's conjecture, and we are not aware of any plausible conjectures for this problem.

\medskip \noindent\textbf{Acknowledgments.}
We thank the anonymous referees for careful readings and helpful comments.


\end{document}